\documentclass{elsarticle}

\usepackage[english]{babel}
\usepackage{dsfont} % \mathds
\usepackage{graphicx}
\usepackage{color}
\usepackage{lineno,hyperref}
\modulolinenumbers[5]
\usepackage{enumerate}
\usepackage{amssymb} % \upharpoonright
\usepackage{amsthm} % \newtheorem
\usepackage{empheq}
\usepackage{cleveref}
\usepackage{xcolor}

\allowdisplaybreaks

\newtheorem{thm}{Theorem}[section]
\newtheorem{lem}[thm]{Lemma}

\theoremstyle{definition}

\newtheorem{dfn}[thm]{Definition}

\newtheorem{exm}[thm]{Example}

\theoremstyle{remark}
\newtheorem{rem}[thm]{Remark}

\newcommand{\exmsymbol}{\hfill$\circ$}

\newcommand{\cset}{\mathds{C}}

\newcommand{\nset}{\mathds{N}}

\newcommand{\rset}{\mathds{R}}

\newcommand{\tset}{\mathds{T}}
\newcommand{\zset}{\mathds{Z}}

\newcommand{\cov}{\mathrm{cov}}

\newcommand{\diff}{\mathrm{d}}
\newcommand{\divv}{\mathrm{div}\,}

\newcommand{\rot}{\mathrm{rot}\,}

\newcommand{\cO}{\mathcal{O}}

\newcommand{\cS}{\mathcal{S}}

\newcommand{\cZ}{\mathcal{Z}}

\newcommand{\otex}[2]{\overset{\parbox{5em}{\scriptsize\centering\text{#1}}}{#2}}

\journal{ArXiv} %J.\ Diff.\ Eq.} % or Commun. Math. Phys.???
\biboptions{sort&compress}
\begin{document}

\begin{frontmatter}

\title{Schwartz Function Valued Solutions of the Euler and the Navier--Stokes Equations}

\author{Philipp J.\ di Dio}
\address{Department of Mathematics and Statistics, University of Konstanz, Universit\"atsstra{\ss}e 10, D-78464 Konstanz, Germany}
\ead{philipp.didio@uni-konstanz.de}

\begin{abstract}
We prove the existence of a solution for the second order system of partial differential equations $\partial_t f = \nu\cdot\Delta f + g\cdot\nabla f + h\cdot f + k$ by a Montel space version of Arzel\`a--Ascoli and bound all Schwartz semi-norms. We find that for the Euler and the Navier--Stokes equations the vorticity remains a Schwartz function as long as the classical solution exists. Our approach is not affected by viscosity. It treats the hyperbolic Euler and the parabolic Navier--Stokes equation simultaneously.
\end{abstract}


\end{frontmatter}

%\linenumbers

%\tableofcontents

\section{Introduction}%%%
%%%%%%%%%%%%%%%%%%%%%%%%%

The motion of (incompressible) fluids in $\rset^n$ or $\tset^n := \rset^n/\zset^n$ ($n=2,3$) are described by the \emph{Euler} ($\nu = 0$) and \emph{Navier--Stokes} ($\nu > 0$) \emph{equations}
\begin{subequations}\label{eq:ens}
\begin{align}
\partial_t u(x,t) &= \nu\Delta u(x,t) - u\cdot\nabla u(x,t) -\nabla p(x,t) + F(x,t)\\
\divv u(x,t) &= 0
\intertext{with initial conditions}
u(x,t_0) &= u_0(x).
\end{align}
\end{subequations}
Here $x\in\rset^n$ or $\tset^n$ is the position vector and $t\geq t_0$ is the time; $t_0\in\rset$ is the initial starting time and without loss of generality $t_0=0$. Then $u(x,t) = (u_1(x,t),\dots,u_n(x,t))^t$ is the velocity field of the fluid, $p(x,t)$ is the pressure, and $F(x,t) = (F_1(x,t),\dots,F_n(x,t))^t$ are externally applied forces  \cite{feffer06}. Reasonable initial conditions \cite{feffer06} are
\[u_0=(u_{0,1},\dots,u_{0,n})^t\in\cS(\rset^n,\rset^n)\quad\text{resp.}\quad C^\infty(\tset^n,\tset^n),\]
i.e., all $u_{0,i}$ shall be Schwartz functions
\[\cS(\rset^n) := \{f\in C^\infty(\rset^n) \,|\, \|x^\alpha\cdot \partial^\beta f(x)\|_\infty<\infty\ \text{for all}\ \alpha,\beta
\in\nset_0^n\},\]
resp.\ smooth periodic functions. A physically reasonable solution $u$ and $p$ of (\ref{eq:ens}) must fulfill the smoothness condition
\[u_1,\dots,u_n,p\in C^\infty(\rset^n\times [0,\infty))\quad\text{resp.}\quad C^\infty(\tset^n\times [0,\infty))\]
and the bounded energy condition
$\int |u(x,t)|^2~\diff x < C$
for all $t\geq t_0$ \cite{feffer06}.

With $F=0$, taking the curl of (\ref{eq:ens}) gives
\begin{equation}\label{eq:ensRot}
\begin{split}
\partial_t \omega(x,t) &= \nu\Delta \omega(x,t) - u(x,t)\cdot\nabla \omega(x,t) + \omega(x,t)\cdot\nabla u(x,t)\\
\omega(x,t_0) &= \omega_0(x) := \rot u_0(x)
\end{split}
\end{equation}
with the \emph{vorticity} $\omega(x,t) := \rot u(x,t)$ ($=\mathrm{curl}\, u(x,t) = \nabla\times u(x,t)$), and we have
\begin{equation}\label{eq:hDef}
\omega\cdot\nabla u = (\omega_1\partial_1 + \omega_2 \partial_2 + \omega_3\partial_3)\!\! \begin{pmatrix} u_1\\ u_2\\ u_3\end{pmatrix}  = \begin{pmatrix}
\partial_1 u_1 & \partial_2 u_1 & \partial_3 u_1\\ \partial_1 u_2 & \partial_2 u_2 & \partial_3 u_2\\ \partial_1 u_3 & \partial_2 u_3 & \partial_3 u_3\end{pmatrix}\!\! \begin{pmatrix} \omega_1\\ \omega_2\\ \omega_3\end{pmatrix} = \nabla u\cdot\omega.
\end{equation}
More on the Euler and the Navier--Stokes equations can be found e.g.\ in 
\cite{euler57,navier27,stokes49,oseen11,oseen27,leray33,leray34a,leray34b,hopf51,ladyzh53,
kisele57,ladyzh63,ladyzh69,caffar82,beale84,temam84,wahl85,constan88,kreiss89,dobrok94,temam95,
temam01,foias01,majda02,lemarie02,darrig02,ladyzh03,brando04,tao06,bahour11,boyer13,
lemari13,robinson16} and references therein.

In this paper we investigate Schwartz function valued vorticity solutions of the Euler and Navier--Stokes equations. Since the initial values fulfill $u_0\in\cS(\rset^3,\rset^3)$ \cite{feffer06} it is interesting if the solution $u$ of (\ref{eq:ens}) resp.\ the vorticity $\omega$ of (\ref{eq:ensRot}) stay in $\cS(\rset^3,\rset^3)$ or how they leave this space.

The spatial decay (asymptotics) of $u$ and $\omega$ has been investigated before. A classical result is that unless the \emph{Dobrokhotov--Shafarevich conditions} \cite{dobrok94} are fulfilled, the spacial decay in $u$ does not decay faster than $O(|x|^{-4})$. This especially covers the instantaneous breakdown of $u$ being a Schwartz function. For further studies see e.g.\ \cite{brando01,brando02,brando04a,bae09,brando09}. For the vorticity $\omega$ such a spreading was never observed, see e.g.\ \cite[Ch.\ 4.11]{lemari13}. In \cite[Prop.\ 3.1]{brando04} it was shown that for the Navier--Stokes equation the vorticity remains a Schwartz function for small times. We also want to mention the works \cite{ting83,danchin96,miyaka00,meyer01,gallay02,mcowen17,sultan20}.

In this paper we show that the vorticity for the Euler and the Navier--Stokes equation remains a Schwartz function as long as the smooth solution exists (\Cref{thm:ensSchwartz}). Our approach is not affected by the Laplace operator and treats the Euler and the Navier--Stokes equations at the same time. It also covers the anisotropic Laplace operator $\nu\cdot\Delta := \nu_1\partial_1^2 + \dots + \nu_n\partial_n^2$ with $\nu = (\nu_1,\dots,\nu_n)\in [0,\infty)^n$.

Let $n,m\in\nset$. In what follows $\|f\|_\infty := \sup_{x\in\rset^n} |f(x)|$ is the supremum-norm on $\rset^n$ and $x^\alpha := x_1^{\alpha_1}\cdots x_n^{\alpha_n}$, $\partial^\alpha := \partial_1^{\alpha_1}\cdots\partial_n^{\alpha_n}$, and $|\alpha| := \alpha_1 + \dots + \alpha_n$ are multi-index notations with $\alpha=(\alpha_1,\dots,\alpha_n)\in\nset_0^n$. By $T^* \leq\infty$ we denote the maximal time a classical solution of a PDE exists, i.e., the classical solution exists for all $t\in [0,T^*)$ and $T^*$ is maximal with this property. We denote by
\[C_b^\infty(\rset^n,\rset^m) := \{f\in C^\infty(\rset^n,\rset^m) \,|\, \|\partial^\alpha f\|_\infty < \infty\ \text{for all}\ \alpha\in\nset_0^n\}\]
the set of all \emph{smooth bounded functions} and by
\[\cS(\rset^n,\rset^m) := \{f\in C^\infty(\rset^n,\rset^m) \,|\, \|x^\alpha\cdot\partial^\beta f(x)\|_\infty < \infty\ \text{for all}\ \alpha,\beta\in\nset_0^n\}\]
we denote the set of all \emph{Schwartz functions}. Here, for functions $f=(f_1,\dots,f_n)^t$  by $|f|$ we denote $|f|(x) := \sqrt{f_1^2(x) + \dots + f_n^2(x)}$ and $\|x^\alpha f(x)\|_\infty := \sup_{x\in\rset^n} |x^\alpha|\cdot |f|(x)$ for all $\alpha\in\nset_0^n$.

We study (\ref{eq:ens}) and (\ref{eq:ensRot}) via the initial value problem
\begin{equation}\label{eq:generalPDE}
\begin{split}
\partial_t f(x,t) &= \nu\Delta f(x,t) +  g(x,t)\cdot\nabla f(x,t) + h(x,t)\cdot f(x,t) + k(x,t)\\
f(x,0) &= f_0(x),
\end{split}
\end{equation}
and the functions $g(x,t) = (g_1(x,t),\dots$, $g_n(x,t))^t$, $h(x,t) = (h_{i,j}(x,t))_{i,j=1}^m$, and $k(x,t) = (k_1(x,t),\dots,k_m(x,t))^t$ with $n,m\in\nset$ are known vector resp.\ matrix functions.

To show the existence of solutions of (\ref{eq:generalPDE}) we split it into the following four simpler parts and glue them together in a Trotter type \cite{trotter59} fashion.

\begin{exm}[heat equation]\label{exm:heat}
Let $\nu>0$ and for all $t>0$ let $\Theta_{\nu,t}(x) := \frac{1}{\sqrt{\pi\nu t}}\cdot\exp\big(-\frac{x^2}{4\nu t}\big)$ be the heat kernel. Then for $f_0\in\cS(\rset)$ the convolution
\begin{equation}\label{eq:heatExact}
f(x,t):= (\Theta_{\nu,t}*f_0)(x) = \int_{y\in\rset} f_0(x-y)\cdot \Theta_{\nu,t}(y) ~\diff y \quad\in C^1([0,\infty),\cS(\rset,\rset))
\end{equation}
solves the initial value problem
\begin{align*}
\partial_t f(x,t) &= \nu\cdot\Delta f(x,t) & \text{on}\ \rset&\times [0,\infty)\\
f(x,0) &= f_0(x) & \text{on}\ \rset&.\tag*{$\circ$}
\end{align*}
\end{exm}

In higher dimensions for the anisotropic Laplace operator $\nu\cdot\Delta = \nu_1\partial_1^2 + \nu_n\partial_n^2$ we set $\Theta_{\nu,t} := \Theta_{\nu_1,t}^{(1)}\cdots \Theta_{\nu_n,t}^{(n)}$ where $\Theta_{\nu_i,t}^{(i)}$ is the one-dimensional heat kernel acting resp.\ depending only on the $x_i$-coordinate.

\begin{exm}[transport equation]\label{exm:tranport}
Let $f_0\in\cS(\rset)$ and $g\in C([0,\infty),\rset)$. Then
\begin{equation}\label{eq:transportExact}
f(x,t) := f_0\left(x+\int_{0}^t g(s)~\diff s\right) \quad\in C^1([0,\infty),\cS(\rset,\rset))
\end{equation}
is a solution of the initial value problem
\begin{align*}
\partial_t f(x,t) &= g(t)\cdot \nabla f(x,t) & \text{on}\ \rset&\times [0,\infty)\\
f(x,0) &= f_0(x) & \text{on}\ \rset&. \tag*{$\circ$}
\end{align*}
\end{exm}

Note, that if $g$ also depends on $x$, then the previous simple formula does not hold but serves for the ansatz
\begin{equation}\label{eq:transportAnsatz}
f(x,t) \approx f_0\left(x + \int_0^t g(x,s)~\diff s\right) \quad\in C^1([0,\infty),\cS(\rset,\rset))
\end{equation}
for small times $t$.

\begin{exm}[``stretching equation'']\label{exm:stretching}
Let $f_0\in\cS(\rset)$ and $h\in C([0,\infty),C(\rset,\rset))$. Then
\begin{equation}\label{eq:matrix}
f(x,t) := \exp\left( \int_0^t h(x,s)~\diff s \right)\cdot f_0(x) \quad\in C^1([0,\infty),\cS(\rset,\rset))
\end{equation}
is a solution of the initial value problem
\begin{align*}
\partial_t f(x,t) &= h(x,t)\cdot f(x,t) & \text{on}\ \rset&\times [0,\infty)\\
f(x,0) &= f_0(x) & \text{on}\ \rset&. \tag*{$\circ$}
\end{align*}
\end{exm}

\begin{exm}[``addition equation'']\label{exm:addition}
Let $k \in C([0,\infty),\cS(\rset))$ and $f_0\in\cS(\rset)$. Then
\begin{equation}\label{eq:additionExact}
f(x,t) = f_0(x) + \int_0^t k(x,s)~\diff s \quad\in C^1([0,\infty),\cS(\rset,\rset))
\end{equation}
solves the initial value problem
\begin{align*}
\partial_t f(x,t) &= k(x,t) & \text{on}\ \rset&\times [0,\infty)\\
f(x,0) &= f_0(x) & \text{on}\ \rset&.\tag*{$\circ$}
\end{align*}
\end{exm}

Each of the initial value problems $\partial_t f = A_i f$, $i=1,\dots,4$, in Examples \ref{exm:heat} to \ref{exm:addition} gives a time-evolution or in higher dimensions at least an approximate time evolution $f(x,t) = E_i(t,t_0)f_0$. In all four time evolutions, when $f_0\in\cS(\rset)$, then also $f=E_i(\,\cdot\,,t_0)f_0\in C([0,\infty),\cS(\rset,\rset))$. Our aim is to show that (\ref{eq:generalPDE}) with $f_0\in\cS(\rset^n,\rset^m)$ also possesses such a Schwartz valued solution and to give explicit bounds on all semi-norms. The four (approximate) solutions $E_i(t,t_0)f_0$ will be glued together in the Trotter fashion \cite{trotter59} by using a Schwartz valued version of the Arzel\`a--Ascoli Theorem (\Cref{lem:arzelaascoli}), i.e.,
\begin{equation*}%\label{eq:trotterLimit}
f(\,\cdot\,,t) = \lim_{N\to\infty} E_4(t,{\textstyle\frac{N-1}{N}}t) \cdots E_1(t,{\textstyle\frac{N-1}{N}}t) E_4({\textstyle\frac{N-1}{N} t,\frac{N-2}{N}}t)\cdots E_1({\textstyle\frac{1}{N}t,0})f_0
\end{equation*}
where the convergence is controlled in the Schwartz space $\cS(\rset^n,\rset^m)$.

The paper is structured as follows. In \Cref{sec:ArzelaAscoli} we give for completeness of the paper the Schwartz function valued version of the Arzel\`a--Ascoli Theorem (\Cref{lem:arzelaascoli}). In \Cref{sec:fNdfn} the family $\{f_N\}_{N\in\nset}$ of approximate solutions of (\ref{eq:generalPDE}) is defined. In \Cref{sec:cover} the cover $\cov_R f$ of a function $f$ is introduced, i.e., $|f|\leq \cov_R f$. It is used in \Cref{sec:main} to prove the main theorem (\Cref{thm:main}). \Cref{thm:main} is applied in \Cref{sec:burgers} to Burgers' equation and in \Cref{sec:ens} to the Euler and the Navier--Stokes equations.

\section{The Schwartz Function Valued Arzel\`a--Ascoli Theorem}%%%
%%%%%%%%%%%%%%%%%%%%%%%%%%%%%%%%%%%%%%%%%%%%%%%%%%%%%%%%%%%%%%%%%%
\label{sec:ArzelaAscoli}

A set $M\subset\cS(\rset^n)$ is bounded if for all $\alpha,\beta\in\nset_0^n$ there are $C_{\alpha,\beta}>0$ with
\[\sup_{f\in M} \|x^\alpha\cdot\partial^\beta f(x)\|_\infty\ \leq\ C_{\alpha,\beta}\ <\ \infty.\]
$\cS(\rset^n)$ is a complete Montel space and every bounded set is relatively compact.

In the proof of the Arzel\`a--Ascoli Theorem it is crucial that the continuous functions are (real- or) complex-valued to apply the Bolzano--Weierstra{\ss} Theorem since $\rset$ and $\cset$ have the Heine--Borel property: Every bounded sequence has a convergent subsequence resp.\ bounded and closed sets are compact. But every Montel space also has the Heine--Borel property, i.e., the classical proof of the Arzel\`a--Ascoli Theorem \cite{arzela82,ascoli83,arzela95}, see e.g.\ \cite[pp.\ 85--86]{yosida68}, can be literally used for $\cS(\rset^n)$. While this was known before, for the sake of completeness of the paper and to make it self-contained we briefly state and prove the result.

\begin{lem}[$\cS(\rset^n)$-valued version of Arzel\`a--Ascoli]\label{lem:arzelaascoli}
Let $n,m\in\nset$, $T>0$, and $\{f_N\}_{N\in\nset}\subset C([0,T],\cS(\rset^n,\rset^m))$. Assume that
\begin{enumerate}[\bfseries (i)]
\item $\sup_{N\in\nset, t\in [0,T]} \|x^\alpha\cdot\partial_x^\beta f_N(x,t)\|_\infty < \infty$ for all $\alpha,\beta\in\nset_0^n$, and

\item $\{f_N\}_{N\in\nset}$ is equi-continuous, i.e., for all $\varepsilon>0$ exists $\delta=\delta(\varepsilon) > 0$ such that for all $N\in\nset$ we have
\[|t-s|<\delta \quad\Rightarrow\quad \|f_N(x,t) - f_N(x,s)\|_\infty \leq \varepsilon.\]
\end{enumerate}
Then $\{f_N\}_{N\in\nset}$ is relatively compact in $C([0,T],\cS(\rset^n,\rset^m))$.
\end{lem}
\begin{proof}
It is sufficient to prove the result for $m=1$. Then it holds in one component of $f_N$ and by choosing subsequences it holds in all components.

Let $\{t_k\}_{k\in\nset}\subset [0,T]$ be a dense countable subset such that for every $\varepsilon>0$ there is a $k(\varepsilon)\in\nset$ with
\[\sup_{t\in [0,T]} \inf_{1\leq k\leq k(\varepsilon)} |t-t_k| \leq \varepsilon.\]

Let $t\in [0,T]$. Since $\{f_N(\,\cdot\,,t)\}_{N\in\nset}$ is a bounded set in the complete Montel space $\cS(\rset^n)$, it has a convergent subsequence. Let $(N_{1,i})_{i\in\nset}\subseteq\nset$ be such that $(f_{N_{1,i}}(\,\cdot\,,t_1))_{i\in\nset}$ converges in $\cS(\rset^n)$. Take a subsequence $(N_{2,i})_{i\in\nset}$ of $(N_{1,i})_{i\in\nset}$ such that $(f_{N_{2,i}}(\,\cdot\,,t_2))_{i\in\nset}$ converges in $\cS(\rset^n)$. Hence, by the diagonal process of choice we get a subsequence $(f_{N_i})_{i\in\nset}$ with $N_i := N_{i,i}$ which converges in $\cS(\rset^n)$ for all $t_k$.

Let $\varepsilon>0$. By the equi-continuity of $\{f_N\}_{N\in\nset}$ there is a $\delta=\delta(\varepsilon)>0$ such that $|t-s|<\delta$ implies $\|f_N(x,t)-f_N(x,s)\|_\infty\leq\varepsilon$. Hence, for every $t\in [0,T]$ there exists a $k$ with $k\leq k(\varepsilon)$ such that
\begin{align*}
\|f_{N_{i}}(x,t) - f_{N_{j}}(x,t)\|_\infty
&\leq \|f_{N_{i}}(x,t) - f_{N_{i}}(x,t_k)\|_\infty + \|f_{N_i}(x,t_k) - f_{N_j}(x,t_k)\|_\infty\\
&\quad + \|f_{N_j}(x,t_k) - f_{N_j}(x,t)\|_\infty\\
&\leq 2\varepsilon + \|f_{N_i}(x,t_k) - f_{N_j}(x,t_k)\|_\infty.
\end{align*}
Thus $\displaystyle\lim_{i,j\to\infty} \sup_{t\in [0,T]} \|f_{N_{i}}(x,t) - f_{N_{j}}(x,t)\|_\infty \leq 2\varepsilon$ and since $\varepsilon>0$ was arbitrary we have $\displaystyle\lim_{i,j\to\infty} \sup_{t\in [0,T]} \|f_{N_{i}}(x,t) - f_{N_{j}}(x,t)\|_\infty =0$. So for every $x\in\rset^n$ the sequence $f_{N_i}(x,\,\cdot\,)$ converges uniformly on $[0,T]$ to a continuous function $f(x,\,\cdot\,)$. Hence by construction $f(\,\cdot\,,t_k)\in\cS(\rset^n)$ for all $t_k$ dense in $[0,T]$. But
\[\|x^\alpha\cdot\partial^\beta f(x,t)\|_\infty \leq \sup_{s\in [0,T], N\in\nset} \|x^\alpha\cdot\partial^\beta f_N(x,s)\|_\infty < \infty\]
for all $\alpha,\beta\in\nset_0^n$ implies $f(\,\cdot\,,t)\in\cS(\rset^n)$ for all $t\in [0,T]$.
\end{proof}

\section{The Approximate Solutions $f_N$}%%%
%%%%%%%%%%%%%%%%%%%%%%%%%%%%%%%%%%%%%%%%%%%%
\label{sec:fNdfn}

\begin{dfn}
Let $N\in\nset$ and $T>0$. A \emph{decomposition $\cZ_N$ of $[0,T]$} is a set $\cZ_N = \{t_0,t_1,\dots,t_N\}$ with
$t_0=0 < t_1 < \dots < t_N = T$
and we set $\Delta\cZ_N := \max_{i=1,\dots,N} |t_{i}-t_{i-1}|$.
\end{dfn}

We use the following type of functions.

\begin{dfn}
Let $d\in\nset_0$, $n\in\nset$, and $T>0$. We denote by
\[C^d([0,T],\cS(\rset^n))\]
all functions $f:\rset^n\times[0,T]\to\rset$ such that
\begin{enumerate}[\bfseries (i)]
\item for every $x\in\rset^n$ we have $f(x,\,\cdot\,)\in C^d([0,T],\rset)$,

\item $f(\,\cdot\,,t),\partial_t f(\,\cdot\,,t),\dots,\partial_t^d f(\,\cdot\,,t)\in\cS(\rset^n)$ for all $t\in [0,T]$, and

\item $\partial_t^i\partial_x^\alpha f \in C(\rset^n\times [0,T],\rset)$ for all $i=0,1,\dots,d$ and $\alpha\in\nset_0^n$.
\end{enumerate}
\end{dfn}

We will always denote by $\partial_t$ the time derivative and we will therefore abbreviate the spatial derivatives $\partial^\alpha = \partial_x^\alpha$. (iii) implies by the Theorem of Schwarz that the order of the derivatives $\partial_t$, $\partial_1$, \dots, $\partial_n$ can be arbitrary.

\begin{dfn}\label{dfn:E1toE4}
Let $n,m\in\nset$, $f_0\in\cS(\rset^n,\rset^m)$, and $t_0\in\rset$. We define the following time evolutions $E_1,\dots,E_4$ for all $t\in [t_0,\infty)$:
\begin{enumerate}[\bfseries (i)]
\item Let $\nu=(\nu_1,\dots,\nu_n)\in [0,\infty)^n$. Then we define
\[E_1(t,t_0)f_0 := \Theta_{\nu,t-t_0} * f_0\]
with $\Theta_{\nu,t} := \Theta_{\nu_1,t}^{(1)}\cdots \Theta_{\nu_n,t}^{(n)}$ where $\Theta_{\nu_i,t}^{(i)}$ is the one-dimensional heat kernel acting resp.\ depending only on the $x_i$-coordinate.

\item For $g=(g_1,\dots,g_n)\in C([t_0,\infty),C_b^\infty(\rset^n,\rset^n))$ we define
\[E_2(t,t_0)f_0 := f_0\left( x + \int_{t_0}^t g(\,\cdot\,,s)~\diff s\right).\]

\item For $h=(h_{i,j})_{i,j=1}^m\in C([t_0,\infty),C_b^\infty(\rset^n,\rset^{m\times m}))$ we define
\[E_3(t,t_0)f_0 := \left(1 + \int_{t_0}^t h(\,\cdot\,,s)~\diff s\right)\cdot f_0.\]

\item For $k=(k_1,\dots,k_m)\in C([t_0,\infty),\cS(\rset^n,\rset^m))$ we define
\[E_4(t,t_0)f_0 := f_0 + \int_{t_0}^t k(\,\cdot\,,s)~\diff s.\]
\end{enumerate}
\end{dfn}

We see that the $E_i$'s are (approximate) time evolutions $\partial_t f = A_i f$ with respect to the operators $A_1 = \nu\cdot\Delta$, $A_2 = g\cdot\nabla$, the multiplication operator $A_3 = h\cdot$, and the addition operator $A_4 = + k$. By the Trotter approach \cite{trotter59} as an approximate solution of $\partial_t f = (A_1 + \dots + A_4)f$ we can therefore take
\[E_4(t_N,t_{N-1}) E_3(t_N,t_{N-1}) E_2(t_N,t_{N-1}) E_1(t_N,t_{N-1}) E_4(t_{N-1},t_{N-2})\cdots E_1(t_1,t_0)f_0.\]

\begin{dfn}\label{dfn:fN}
Let $n,m\in\nset$, $T>0$, $f_0\in\cS(\rset^n,\rset^m)$, and let $\nu$, $g$, $h$, and $k$ be as in \Cref{dfn:E1toE4}. For each $N\in\nset$ let $\cZ_N$ be a decomposition of $[0,T]$. We define the functions $f_N:\rset^n\times [0,T]\to\rset^m$ piece-wise on each interval $[t_i,t_{i+1}]$ by the following:
\begin{enumerate}[\bfseries (i)]
\item For $t\in [t_0,t_1]$ we set
\[f_N(\,\cdot\,,t) := E_4(t,t_0) E_3(t,t_0) E_2(t,t_0) E_1(t,t_0) f_0.\]

\item For $t\in [t_i,t_{i+1}]$ with $i=1,\dots,N-1$ we set
\[f_N(\,\cdot\,,t) := E_4(t,t_i) E_3(t,t_i) E_2(t,t_i) E_1(t,t_i) f_N(\,\cdot\,,t_i).\]
\end{enumerate}
\end{dfn}

The following ensures that we can apply the Arzel\`a--Ascoli Theorem to the family $\{f_N\}_{N\in\nset}$.

\begin{lem}\label{lem:fNschwartz}
$f_N\in C([0,T],\cS(\rset^n,\rset^m))$ and $C^1$ in $t\in (t_i,t_{i+1})$. %for all $N\in\nset$.
\end{lem}
\begin{proof}
For all $f\in\cS(\rset^n,\rset^m)$ and $t \geq t'\geq 0$ we have $E_j(t,t')f\in\cS(\rset^n,\rset^m)$, $j=1,\dots,4$, and from \Cref{dfn:E1toE4} we also have $C^1$ in $t\in (t_i,t_{i+1})$.
\end{proof}

\section{The Cover of a Function}%%%
%%%%%%%%%%%%%%%%%%%%%%%%%%%%%%%%%%%%
\label{sec:cover}

For the approximate solutions $f_N$'s we need to bound all semi-norms
\[\|x^\alpha\cdot \partial^\beta f_N(x,t)\|_\infty\]
to apply the Arzel\`a--Ascoli Theorem. To handle these extensive calculations, we introduce the cover $\cov_R f$ of a function $f$, i.e., $|f_N(\,\cdot\,,t)|\leq \cov_R f_N(\,\cdot\,,t)$.

\begin{dfn}\label{dfn:cov}
Let $n,m\in\nset$ and $R\geq 0$. We introduce the \emph{cover $\cov_R$ of a Schwartz function} $f\in\cS(\rset^n,\rset^m)$ as the map
\[\cov_R:\cS(\rset^n,\rset^m)\to C(\rset^n,\rset)\]
defined by
\[(\cov_R f)(x):= \max_{y\in\rset^n:\|y\|_2\geq \|x\|_2 - R} \|f(y)\|_2.\]
\end{dfn}

The cover has the following general properties.

\begin{lem}\label{lem:coverGeneral}
Let $n,m\in\nset$, $R,R'\geq 0$, $a,b\in\rset$, and $f,f'\in\cS(\rset^n,\rset^m)$. The cover $\cov$ has the following general properties:
\begin{enumerate}[\bfseries (i)]
\item $\cov_R f$ is non-negative, radial symmetric, decreases with increasing $\|x\|_2$, and fulfills $|f| \leq \cov_0 f$.

\item If $|f|\leq |f'|$, then $\cov_R f \leq \cov_R f'$.

\item $\cov_R f \leq \cov_{R+R'} f$.

\item $\cov_R (af+bf') \leq |a|\cdot\cov_R f + |b|\cdot\cov_R f'$.

\item $\cov_R (\cov_{R'} f) = \cov_{R+R'} f$.

\item For $d\in\nset$ let $C_d>0$ be such that $|f|(x)\leq \frac{C_d}{1+\|x\|_2^d}$ for all $x\in\rset^n$. Then
\[(\cov_0 f)(x) \leq \frac{C_d}{1+\|x\|_2^d}\]
for all $x\in\rset^n$.

\item For $d\in\nset$ let $C_d>0$ be such that $|f(x)|\leq \frac{C_d}{1+\|x\|_2^d}$ for all $x\in\rset^n$. Then
\[(\cov_R f)(x) \leq \frac{C_d}{1+r^d} \qquad\text{with}\qquad r := \max\{0,\|x\|_2-R\}\]
for all $x\in\rset^n$.
\end{enumerate}
\end{lem}
\begin{proof}
(i)-(iv): Follows immediately from the \Cref{dfn:cov}.

(v): By (i) we have that $\cov_{R'} f$ is radial symmetric and decreases with increasing $\|x\|_2$. Hence, for fixed $x\in\rset^n$ there is a $y\in\rset^n$ with $\|y'\|_2\geq \|x\|_2-R'$ such that $|f(y')| = (\cov_{R'} f)(x)$. But in $\cov_{R+R'} f$ we have the restriction $\|y\|_2\geq \|x\|_2 - R - R'$, i.e., a larger range for $y$ and hence the inequality holds.

(vi): Since $a_d(x) := \frac{C_d}{1+\|x\|_2^d}$ is non-negative, radial symmetric, and decreases with increasing $\|x\|_2$ we have $\cov_0 a_d = a_d$ and since $|f|\leq a_d$ we have that (ii) implies $\cov_0 f \leq \cov_0 a_d = a_d$.

(vii): Set $a_d(x) := \frac{C_d}{1+\|x\|_2^d}$. Then
\[(\cov_R f)(x) \overset{\text{(v)}}{=} \cov_R (\cov_0 f) \overset{\text{(vi)}}{\leq} \cov_R a_d\]
and since $a_d$ is non-negative, radial symmetric, decreases with increasing $\|x\|_2$ we have $\cov_R a_d = \frac{C_d}{1+r^d}$ with $r := \max\{0,\|x\|_2-R\}$.
\end{proof}

We use the cover $\cov$ to bound the Schwartz semi-norms for the approximate time evolution from $E_1,\dots,E_4$. Hence, we collect in the following the special properties connected to the $E_i$'s.

\begin{lem}\label{lem:coverSpecialE}
Let $n,m\in\nset$, $R,R'\geq 0$, $t_0,t\in\rset$ with $t\geq t_0$, and $f\in\cS(\rset^n,\rset^m)$. Let $\nu$, $g$, $h$, and $k$ be as in \Cref{dfn:E1toE4}. The cover $\cov$ has the following special properties:
\begin{enumerate}[\bfseries (i)]
\item $\cov_R (E_1(t,t_0)f) \leq E_1(t,t_0) \cov_R f$.

\item With $G(t,t_0) := \int_{t_0}^t \|g(\,\cdot\,,s)\|_\infty~\diff s$ we get
\[|E_2(t,t_0)f| \leq \cov_{G(t,t_0)} f.\]

\item With $H(t,t_0) := \int_{t_0}^t \|h(\,\cdot\,,s)\|_\infty~\diff s$ we have
\[ \cov_R (E_3(t,t_0) f) \leq (1+H(t,t_0))\cdot\cov_R f.\]

\item $\cov_R (E_4(t,t_0)f) \leq \cov_R f + \int_{t_0}^t \cov_R k(\,\cdot\,,s)~\diff s$.
\end{enumerate}
\end{lem}
\begin{proof}
(i): Follow immediately from the fact that $E_1(t,t_0)$ is the convolution with the non-negative heat kernel $\Theta_{\nu,t-t_0}$.

(ii): Since $E_2(t,t_0)$ is a translation, each point $x\in\rset^n$ with $\|x\|_2=r$ is moved to some $x'\in\rset^n$ with $\|x'\|_2 \leq r + G(t,t_0)$ which proves the inequality.

(iii): Follows immediately from taking the supremum of $1 + \int_{t_0}^t\! h(\,\cdot\,,s)\,\diff s$.

(iv): We have
\begin{align*}
\cov_R (E_4(t,t_0)f)
&\overset{\phantom{\text{Lem.\ \ref{lem:coverGeneral}(iv)}}}{=} \cov_R\left( f + \int_{t_0}^t k(\,\cdot\,,s)~\diff s\right)\\
&\overset{\text{Lem.\ \ref{lem:coverGeneral}(iv)}}{\leq} \cov_R f + \cov_R\left( \int_{t_0}^t k(\,\cdot\,,s)~\diff s\right)\\
&\overset{\text{Lem.\ \ref{lem:coverGeneral}(iv)}}{\leq} \cov_R f + \int_{t_0}^t \cov_R k(\,\cdot\,,s)~\diff s.\tag*{\qedhere}
\end{align*}
\end{proof}

We want to bound all semi-norms in the Schwartz space. Hence, we also have to look at derivatives of the approximate solutions $f_N$.

\begin{lem}\label{lem:fxtDerivatives}
Let $n,m\in\nset$, $d\in\nset_0$, and $f_0\in\cS(\rset^n,\rset^m)$. Let $\nu$, $g$, $h$, and $k$ be as in \Cref{dfn:E1toE4}. Set
\begin{equation}\label{eq:fxt}\begin{split}
f(x,t) &:= E_4(t,0) E_3(t,0) E_2(t,0) E_1(t,0) f_0\\
&\phantom{:}= \left(\!1 +\! \int_0^t\!\! h(x,s)\,\diff s\right)(\Theta_{\nu,t} * f_0)\left(\!x +\! \int_0^t\!\! g(x,s)\,\diff s\right) + \int_0^t\!\! k(x,s)\,\diff s\\
&\phantom{:}= (1+H)\cdot F(x+G) + K
\end{split}
\end{equation}
with $F:=\Theta_{\nu,t}*f_0$, $H := \int_0^t h(\,\cdot\,,s)~\diff s$, $G := \int_0^t g(\,\cdot\,,s)~\diff s$, $G_j := \int_0^t g_j(\,\cdot\,,s)~\diff s$, and $K :=\int_0^t k(\,\cdot\,,s)~\diff s$. Then for $i_1,\dots,i_d\in\{1,\dots,n\}$ we have
\begin{align}\label{eq:fxtDerivatives}
\partial_{i_d}\dots\partial_{i_1} f(x,t) &= \bigg(1+H+ \sum_{r=1}^{d} \partial_{i_r}G_{i_r}\bigg)\cdot (\partial_{i_{d}}\dots \partial_{i_1} F)(x+G) + \partial_{i_{d}}\dots\partial_{i_1} K\notag\\
&\quad + \sum_{r=1}^{d} \sum_{j_r\neq i_r} \partial_{i_r}G_{j_r} \cdot (\partial_{i_{d}}\dots\partial_{i_{r+1}}\partial_{j_r}\partial_{i_{r-1}}\dots\partial_{i_1} F)(x+G)\\
&\quad + \cO(t^2) + o(d-1,t)\notag
\end{align}
where $o(d-1,t)$ is the set of functions $\int_0^t \{\text{derivatives of}\ h\ \text{or}\ g\}\cdot \Theta_{\nu,t}*(\partial^\gamma f_0)$ with $\gamma\in\nset_0^n$ and $|\gamma|\leq d-1$, i.e., the growth in $t$ is at most linear.
\end{lem}
\begin{proof}
We prove (\ref{eq:fxtDerivatives}) by induction over $d\in\nset_0$.

\underline{$d=0$:} Clear, since (\ref{eq:fxt}) = (\ref{eq:fxtDerivatives}).

\underline{$d\to d+1$:} Assume (\ref{eq:fxtDerivatives}) holds for some $d\in\nset_0$. Let $i_1,\dots,i_{d+1}\in\{1,\dots,n\}$. Then we have
\begin{align*}
&\partial_{i_{d+1}}\dots\partial_{i_1} F\\
&= \partial_{i_{d+1}} [\partial_{i_d}\dots\partial_{i_1} F]\\
&= \partial_{i_{d+1}} \bigg[ \bigg(1+H+ \sum_{r=1}^d\partial_{i_r}G_{i_r}\bigg)\cdot (\partial_{i_d}\dots \partial_{i_1} F)(x+G)\\
&\qquad\qquad + \sum_{r=1}^d\sum_{j_r\neq i_r} \partial_{i_r}G_{j_r}\cdot (\partial_{i_d}\dots\partial_{i_{r+1}}\partial_{j_r}\partial_{i_{r-1}}\dots\partial_{i_1} F)(x+G)\\
&\qquad\qquad + \partial_{i_d}\dots\partial_{i_1} K + \cO(t^2) + o(d-1,t) \bigg]\\
&= \underbrace{\bigg(\partial_{i_{d+1}} H+ \sum_{r=1}^d \partial_{i_{d+1}}\partial_{i_r}G_{i_r}\bigg)\cdot (\partial_{i_d}\dots \partial_{i_1} F)(x+G)}_{\in o(d,t)}\\
&\quad + \bigg(1+H+ \sum_{r=1}^d\partial_{i_r}G_{i_r}\bigg)\cdot \sum_{j_{d+1}=1}^n (\delta_{j_{d+1},i_{d+1}} + \partial_{i_{d+1}}G_{j_{d+1}})\\
&\qquad\qquad\qquad\qquad\qquad\qquad\times (\partial_{j_{d+1}}\partial_{i_d}\dots \partial_{i_1} F)(x+G)\\
&\quad + \sum_{r=1}^d\sum_{j_r\neq i_r} \underbrace{\partial_{i_{d+1}}\partial_{i_r}G_{j_r}\cdot (\partial_{i_d}\dots\partial_{i_{r+1}}\partial_{j_r}\partial_{i_{r-1}}\dots\partial_{i_1} F)(x+G)}_{\in o(d,t)}\\
&\quad + \sum_{r=1}^d \sum_{j_r\neq i_r} \underbrace{\partial_{i_r}G_{j_r}\cdot \sum_{j_{d+1}=1}^n (\delta_{j_{d+1},i_{d+1}} + \partial_{i_{d+1}} G_{j_{d+1}})}_{\substack{= \partial_{i_r} G_{j_r}\cdot (1+\partial_{i_{d+1}} G_{i_{d+1}})= \partial_{i_r} G_{j_r} + \cO(t^2)\ \text{for}\ j_{d+1}=i_{d+1};\\ = \partial_{i_r} G_{j_r}\cdot \partial_{i_{d+1}} G_{j_{d+1}}\in\cO(t^2)\qquad\qquad\quad \text{for}\ j_{d+1}\neq i_{d+1}}}\\
&\qquad\qquad\qquad\qquad\qquad\qquad \times (\partial_{j_{d+1}}\dots\partial_{i_{r+1}}\partial_{j_r}\partial_{i_{r-1}}\dots\partial_{i_1} F)(x+G)\\
&\quad + \partial_{i_{d+1}}\dots\partial_{i_1} K + \cO(t^2) + o(d-1,t)\\
&= \underbrace{\bigg(1+H+ \sum_{r=1}^d\partial_{i_r}G_{i_r}\bigg)\cdot (1 + \partial_{i_{d+1}}G_{i_{d+1}})}_{= 1+H+ \sum_{r=1}^{d+1} \partial_{i_r}G_{i_r} + \cO(t^2)}\cdot (\partial_{i_{d+1}}\dots \partial_{i_1} F)(x+G)\\
&\quad + \sum_{j_{d+1}\neq i_{d+1}}^n \underbrace{\bigg(1+H+ \sum_{r=1}^d\partial_{i_r}G_{i_r}\bigg)\cdot\partial_{i_{d+1}}G_{j_{d+1}}}_{=\partial_{i_{d+1}}G_{j_{d+1}}+\cO(t^2)}\cdot (\partial_{j_{d+1}}\partial_{i_d}\dots \partial_{i_1} F)(x+G)\\
&\quad + \sum_{r=1}^d \sum_{j_r\neq i_r} \partial_{i_r}G_{j_r} \cdot (\partial_{i_{d+1}}\dots\partial_{i_{r+1}}\partial_{j_r}\partial_{i_{r-1}}\dots\partial_{i_1} F)(x+G)\\
&\quad + \partial_{i_{d+1}}\dots\partial_{i_1} K + \cO(t^2) + o(d,t)\\
&= \bigg(1+H+ \sum_{r=1}^{d+1} \partial_{i_r}G_{i_r}\bigg)\cdot (\partial_{i_{d+1}}\dots \partial_{i_1} F)(x+G)\\
&\quad + \sum_{r=1}^{d+1} \sum_{j_r\neq i_r} \partial_{i_r}G_{j_r} \cdot (\partial_{i_{d+1}}\dots\partial_{i_{r+1}}\partial_{j_r}\partial_{i_{r-1}}\dots\partial_{i_1} F)(x+G)\\
&\quad + \partial_{i_{d+1}}\dots\partial_{i_1} K + \cO(t^2) + o(d,t)
\end{align*}
which proves (\ref{eq:fxtDerivatives}) for $d+1$ and hence by induction (\ref{eq:fxtDerivatives}) for all $d\in\nset_0$.
\end{proof}

\section{The Existence of a Schwartz Function Valued Solution}%%%
%%%%%%%%%%%%%%%%%%%%%%%%%%%%%%%%%%%%%%%%%%%%%%%%%%%%%%%%%%%%%%%%%
\label{sec:main}

The next result is the main theorem of this article. It completely solves the Schwartz function regularity problem of (\ref{eq:generalPDE}) with explicit bounds for the Schwartz function semi-norms $\|x^\alpha\cdot \partial^\beta f(x,t)\|_\infty$

\begin{thm}\label{thm:main}
Let $n,m\in\nset$, $d\in\nset_0$, and $\nu=(\nu_1,\dots,\nu_n)\in [0,\infty)^n$. Furthermore, let $g\in C^d([0,\infty),C_b^\infty(\rset^n,\rset^n))$, $h\in C^d([0,\infty),C_b^\infty(\rset^n,\rset^{m\times m}))$, and $k\in C^d([0,\infty),\cS(\rset^n,\rset^m))$. Set $H(t_1,t_0) := \int_{t_0}^{t_1} \|h(\,\cdot\,,s)\|_\infty\,\diff s$, $G(t_1,t_0) := \int_{t_0}^{t_1} \|g(\,\cdot\,,s)\|_\infty\,\diff s$, and $G'(t_1,t_0) := \int_{t_0}^{t_1} \|\nabla g(\,\cdot\,,s)\|_\infty\,\diff s$ for all $t_1 \geq t_0 \geq 0$.
For $f_0\in\cS(\rset^n,\rset^m)$ the initial value problem
\begin{align*}
\partial_t f &= \nu\cdot\Delta f + g\cdot\nabla f + h\cdot f + k\\
f(\,\cdot\,,0) &= f_0
\end{align*}
has a solution $f\in C^{d+1}([0,\infty),\cS(\rset^n,\rset^m))$ with the covers
\begin{equation}\label{eq:mainBound}\begin{split}
|\partial^\beta f(\,\cdot\,,t)| \;\leq\;  B_{|\beta|}(\,\cdot\,,t) := &\ \exp\left(H(t,0) + b^2\cdot G'(t,0) \right)\\
&\qquad\quad\times E_1(t,0)\cov_{G(t,0)}(\max_\beta\partial^\beta f_0)\\
& + \int_0^t \exp\left(H(t,s) + b^2\cdot G'(t,s) \right)\\
&\qquad\quad\times E_1(t,s)\cov_{G(t,s)}\left(\max_\beta \partial^\beta k(\,\cdot\,,s)\right)\diff s\\
& + \overline{B}_{d-1}(\,\cdot\,,t)
%
%
%\exp(H(t,0) + |\beta|^2\cdot G'(t,0))\cdot E_1(t,0)\cov_{G(t,0)} f_0\\
%&\ + \int_0^t H(t,s)\cdot E(t,s) \cov_{G(t,s)} k(\,\cdot\,,s)\;\diff s\\
%&\ + \overline{B}_{|\beta|}(\,\cdot\,,t)
\end{split}\end{equation}
for all $\beta\in\nset_0$ such that
\[\|x^\alpha\cdot \partial^\beta f(x,t)\|_\infty \quad\leq\quad  \sup_{x\in\rset^n} \|x\|_2^{|\alpha|}\cdot B_{|\beta|}(x,t) \quad<\quad\infty\]
for all $\alpha,\beta\in\nset_0^n$ and $t\geq 0$. Therein, $\overline{B}_{d}(\,\cdot\,,t)$ is a linear combination of bounds $B_j$ for $j \leq d-1$ with coefficients as integrals over $\partial^\gamma g$ and $\partial^\gamma h$ with $|\gamma|\leq d$.
\end{thm}
\begin{proof}
Let $N\in\nset$, $T>0$, and $\Delta\cZ_N$ be a decomposition of $[0,T]$. Take the $f_N$'s from \Cref{dfn:fN}, i.e., $f_N\in C([0,T],\cS(\rset^n,\rset^m)$ by \Cref{lem:fNschwartz}.

We look at the family $\{f_N\}_{N\in\nset}$ and want to use \Cref{lem:arzelaascoli} to find an accumulation point $f$ for $N\to\infty$ with $\Delta\cZ_N\to 0$. Since $T>0$ is arbitrary, it is sufficient that we bound the semi-norms at $t=T$. Additionally, since $|\partial^\beta f_N|(\,\cdot\,,T) \leq \cov_0 (\partial^\beta f_N)(\,\cdot\,,T)$ we have
\[\|x^\alpha\cdot \partial^\beta f_N(x,T)\|_\infty \leq \|x^\alpha\cdot \cov_0(\partial^\beta f_N)(x,T)\|_\infty.\]
We will proceed via induction over $b = |\beta|\in\nset_0$.

\underline{$b=0$:} We have
\begin{align}
&|f_N(\,\cdot\,,t_N)|\notag\\
&\otex{Lem.\ \ref{lem:coverGeneral}(i)}{\leq} \cov_0 f_N(\,\cdot\,,t_N)\notag\\
&\otex{Dfn.\ \ref{dfn:fN}}{=} \cov_0 E_4(t_N,t_{N-1}) \dots E_1(t_N,t_{N-1}) f_N(\,\cdot\,,t_{N-1})\notag\\
&\otex{Lem.\ \ref{lem:coverSpecialE}(iv)}{\leq} \cov_0 E_3(t_N,t_{N-1}) \dots E_1(t_N,t_{N-1}) f_N(\,\cdot\,,t_{N-1})\notag\\
&\qquad\qquad\quad\; + \int_{t_{N-1}}^{t_N} \cov_0 k(\,\cdot\,,s)~\diff s\notag\\
&\otex{Lem.\ \ref{lem:coverSpecialE}(iii)}{\leq} \exp(H(t_N,t_{N-1}))\cdot \cov_0 E_2(t_N,t_{N-1}) E_1(t_N,t_{N-1}) f_N(\,\cdot\,,t_{N-1})\notag\\
&\qquad\qquad\quad\; + \int_{t_{N-1}}^{t_N} \cov_0 k(\,\cdot\,,s)~\diff s\notag\\
&\otex{Lem.\ \ref{lem:coverSpecialE}(ii)}{\leq} \exp(H(t_N,t_{N-1}))\cdot \cov_{G(t_N,t_{N-1})} E_1(t_N,t_{N-1}) f_N(\,\cdot\,,t_{N-1})\notag\\
&\qquad\qquad\quad\; + \int_{t_{N-1}}^{t_N} \cov_0 k(\,\cdot\,,s)~\diff s\notag\\
&\otex{Lem.\ \ref{lem:coverSpecialE}(i)}{\leq} \exp(H(t_N,t_{N-1}))\cdot E_1(t_N,t_{N-1})\cov_{G(t_N,t_{N-1})} f_N(\,\cdot\,,t_{N-1})\notag\\
&\qquad\qquad\quad\; + \int_{t_{N-1}}^{t_N} \cov_0 k(\,\cdot\,,s)~\diff s\notag\\
\intertext{and applying \Cref{dfn:fN} and \Cref{lem:coverSpecialE} (i-iv) on the time interval $[t_{N-1},t_{N-2}]$ in the same way gives}
&\otex{}{\leq} \exp(H(t_N,t_{N-2}))\cdot E_1(t_N,t_{N-2})\cov_{G(t_N,t_{N-2})} f_N(\,\cdot\,,t_{N-2})\notag\\
&\qquad\qquad\quad\; + \exp(H(t_N,t_{N-1}))\cdot\int_{t_{N-2}}^{t_{N-1}} E_1(t_N,t_{N-1})\cov_0 k(\,\cdot\,,s)~\diff s\notag\\
&\qquad\qquad\quad\; + \int_{t_{N-1}}^{t_N} \cov_0 k(\,\cdot\,,s)~\diff s\notag
\intertext{and proceeding gives finally}
&\otex{}{\leq} \exp(H(t_N,t_0))\cdot E_1(t_N,t_0)\cov_{G(t_N,t_0)} f_0\notag\\
&\qquad\qquad\quad\; + \sum_{i=0}^N \exp(H(t_N,t_{N-i}))\cdot\int_{t_{N-1-i}}^{t_{N-i}} E_1(t_N,t_{N-i})\cov_{G(t_N,t_{N-i})} k(\,\cdot\,,s)~\diff s\notag
\intertext{which converges by Riemann integration for $N\to\infty$ with $\Delta\cZ_N\to 0$ to}
&\otex{}{\to} B_0(x,T) = \exp(H(T,0))\cdot E_1(T,0)\cov_{G(T,0)} f_0 \label{eq:proof}\\
&\qquad\qquad\qquad\qquad\qquad\; + \int_0^T \exp(H(T,s))\cdot E(T,s) \cov_{G(T,s)} k(\,\cdot\,,s)~\diff s.\notag
\end{align}
By \Cref{lem:coverGeneral}(vi) and (vii) we have that
\[\|x^\alpha\cdot f_N(x,T)\|_\infty \leq \sup_{x\in\rset^n} \|x\|_2^{|\alpha|}\cdot B_0(x,T) < \infty\]
for all $\alpha\in\nset_0^n$.

\underline{$d\to d+1$:} Assume for all $i=0,\dots,d$ we have bounds $B_i$ with
\[|\partial^\beta f_N(x,T)| \leq B_{|\beta|}(x,T)\]
such that
\[\|x^\alpha\cdot \partial^\beta f_N(x,T)\|_\infty \leq \sup_{x\in\rset^n} \|x\|_2^{|\alpha|}\cdot B_{|\beta|}(x,T) < \infty\]
for all $N\in\nset$ and all $\alpha,\beta\in\nset_0^n$ with $|\beta|\leq d$. We show that such a bound $B_{d+1}$ also exists.

From \Cref{lem:fxtDerivatives} we have that $\partial^\beta f(\,\cdot\,,t_N)$ gives an induction from $t_{i}$ to $t_{i-1}$ where the derivative $\partial^\beta$ applies to $f_N$ again, gives contributions of order $\Delta\cZ_N^2$, contributions from lower derivatives linear in $\Delta\cZ_N$. In the limit $\Delta\cZ_N\to 0$ by Riemann integration (\ref{eq:proof}) the $\Delta\cZ_N^2$ contributions vanish and the $o(d,\Delta\cZ_N)$ contributions become a sum over the covers $B_0,\dots, B_d$. Hence, to shorten the calculations, we only drag $\cO(\Delta\cZ_N^2)$ and $o(d,\Delta\cZ_N)$ through the calculations:
\begin{align*}
&\max_{\beta\in\nset_0^n:|\beta|=b} |\partial^\beta f_N(\,\cdot\,,t_N)|\\
&\otex{Lem.\ \ref{lem:coverGeneral}(i)}{\leq} \cov_0 (\max_\beta \partial^\beta f_N(\,\cdot\,,t_N))\\
&\otex{Def.\ \ref{dfn:fN}}{=} \cov_0 (\max_\beta \partial^\beta E_4(t_N,t_{N-1})\dots E_1(t_N,t_{N-1}) f_N(\,\cdot\,,t_{N-1}))\\
&\otex{Lem.\ \ref{lem:coverSpecialE}(iv)}{\leq} \cov_0 (\max_\beta\partial^\beta E_3(t_N,t_{N-1})\dots E_1(t_N,t_{N-1}) f_N(\,\cdot\,,t_{N-1}))\\
&\qquad\qquad\quad\; + \int_{t_{N-1}}^{t_N} \cov_0(\max_\beta\partial^\beta k(\,\cdot\,,s))~\diff s \\
&\otex{Lem.\ \ref{lem:fxtDerivatives}}{\leq} (1+H(t_N,t_{N-1}) + b\cdot G'(t_N,t_{N-1}))\\
&\qquad\qquad\qquad\qquad\times \cov_0 E_2(t_N,t_{N-1}) E_1(t_N,t_{N-1}) (\max_\beta \partial^\beta f_N(\,\cdot\,,t_{N-1}))\\
&\qquad\qquad\quad\; + b(b-1)\cdot G'(t_N,t_{N-1})\\
&\qquad\qquad\qquad\qquad\times \cov_0 E_2(t_N,t_{N-1}) E_1(t_N,t_{N-1})(\max_\beta \partial^\beta f_N(\,\cdot\,,t_{N-1}))\\
&\qquad\qquad\quad\; + \int_{t_{N-1}}^{t_N} \cov_0(\max_\beta\partial^\beta k(\,\cdot\,,s))~\diff s + \cO(\Delta\cZ_N^2) + o(d,\Delta\cZ_N)\\
&\otex{$b\!+\!b(b\!-\!1)\!=\!b^2$}{=} (1+H(t_N,t_{N-1}) + b^2\cdot G'(t_N,t_{N-1}))\\
&\qquad\qquad\qquad\qquad\times \cov_0 E_2(t_N,t_{N-1}) E_1(t_N,t_{N-1}) (\max_\beta \partial^\beta f_N(\,\cdot\,,t_{N-1}))\\
&\qquad\qquad\quad\; + \int_{t_{N-1}}^{t_N} \cov_0(\max_\beta\partial^\beta k(\,\cdot\,,s))~\diff s + \cO(\Delta\cZ_N^2) + o(d,\Delta\cZ_N)\\
&\otex{Lem.\ \ref{lem:coverSpecialE}}{\leq} (1+ H(t_N,t_{N-1}) + b^2\cdot G'(t_N,t_{N-1}))\\
&\qquad\qquad\qquad\qquad\times E_1(t_N,t_{N-1})\cov_{G(t_N,t_{N-1})}(\max_\beta \partial^\beta f_N(\,\cdot\,,t_{N-1}))\\
&\qquad\qquad\quad\; + \int_{t_{N-1}}^{t_N} \cov_0(\max_\beta\partial^\beta k(\,\cdot\,,s))~\diff s + \cO(\Delta\cZ_N^2) + o(d,\Delta\cZ_N)\\
&\otex{$1+y\leq e^y$}{\leq} \exp(H(t_N,t_{N-1}) + b^2\cdot G'(t_N,t_{N-1}))\\
&\qquad\qquad\qquad\qquad\times E_1(t_N,t_{N-1})\cov_{G(t_N,t_{N-1})}(\max_\beta \partial^\beta f_N(\,\cdot\,,t_{N-1}))\\
&\qquad\qquad\quad\; + \int_{t_{N-1}}^{t_N} \cov_0(\max_\beta\partial^\beta k(\,\cdot\,,s))~\diff s + \cO(\Delta\cZ_N^2) + o(d,\Delta\cZ_N)
\intertext{and proceeding with this on each interval $[t_i,t_{i-1}]$ we finally get}
&\otex{}{\leq}  \exp\left(H(t_N,t_0) + b^2\cdot G'(t_N,t_0) \right)\\
&\qquad\qquad\qquad\qquad\times E_1(t_N,t_0)\cov_{G(t_N,t_0)}(\max_\beta\partial^\beta f_N(\,\cdot\,,t_0))\\
&\qquad\qquad\quad\; +\sum_{i=1}^N \exp\left(H(t_N,t_{N+1-i}) + b^2\cdot G'(t_N,t_{N+1-i}) \right)\\
&\qquad\qquad\qquad\qquad\times \int_{t_{N-i}}^{t_{N+1-i}} E_1(t_N,t_{N+1-i})\cov_{G(t_N,t_{N+1-i})}(\max_\beta \partial^\beta k(\,\cdot\,,s))\diff s\\
&\qquad\qquad\quad\; + \cO(\Delta\cZ_N^2) + o(d,\Delta\cZ_N)
\intertext{which converges for $N\to\infty$ with $\Delta\cZ_N\to 0$ by Riemann integration to}
&\otex{}{\to} \exp\left(H(T,0) + b^2\cdot G'(T,0) \right)\cdot E_1(T,0)\cov_{G(T,0)}(\max_\beta\partial^\beta f_0)\\
&\qquad\qquad\quad\; + \int_0^T \exp\left(H(T,s) + b^2\cdot G'(T,s) \right)\\
&\qquad\qquad\qquad\qquad\times E_1(T,s)\cov_{G(T,s)}(\max_\beta \partial^\beta k(\,\cdot\,,s))\diff s\\
&\qquad\qquad\quad\; + \overline{B}_d(\,\cdot\,,T)
\end{align*}
which is the bound $B_{d+1}$ with
\[\|x^\alpha\cdot \partial^\beta f_N(\,\cdot\,,T)\|_\infty \leq \|x^\alpha\cdot B_{d+1}(x)\|_\infty < \infty\]
for all $\alpha,\beta\in\nset_0^n$ with $|\beta|=d+1$. In summary, we have shown that for the family $\{f_N\}_{N\in\nset}$ on $[0,T]$ \Cref{lem:arzelaascoli}(i) is fulfilled.

It remains to show that condition (ii) of \Cref{lem:arzelaascoli} is fulfilled. Since all $f_N$ are piece-wise differentiable, it is sufficient to show that $\partial_t f_N$ is bounded. But this follows immediately from the bounds $B_0,B_1,\dots$ and hence there exists a constant $L>0$ such that
\[\sup_{t\in [0,T]} \|\partial_t f_N(\,\cdot\,,t)\|_\infty \leq L < \infty\]
for all $N\in\nset$, i.e., $f_N$ are all Lipschitz in $t\in [0,T]$ with a Lipschitz constant independent on $N$, $x$, and $t$. Condition (ii) in \Cref{lem:arzelaascoli} is then fulfilled since Lipschitz continuity implies equi-continuity.

Since conditions (i) and (ii) of \Cref{lem:arzelaascoli} are fulfilled $\{f_N\}_{N\in\nset}$ is relatively compact. Hence, there exists a subsequence $(N_i)_{i\in\nset}\subseteq\nset$ such that $f_{N_i}$ converges on $\rset^n\times [0,T]$ to a function $f\in C([0,T],\cS(\rset^n,\rset^m))$, i.e.,
\begin{equation}\label{eq:convergenz}
\sup_{t\in [0,T]} \big\|x^\alpha\cdot\partial^\beta f_{N_i}(x,t)-x^\alpha\cdot\partial^\beta f(x,t)\big\|_\infty\ \xrightarrow{i\to\infty}\ 0
\end{equation}
for all $\alpha,\beta\in\nset_0^n$.

We now show that the accumulation point $f$ solves (\ref{eq:generalPDE}) and is in $C^{d+1}$ in $t$, i.e., $f\in C^{d+1}([0,T],\cS(\rset^n,\rset^m))$. By \Cref{dfn:fN} of the $f_N$ we have that each $f_N$ is piece-wise differentiable in $t$ and taking the derivative $\partial_t f_N$ we find
\begin{multline}\label{eq:proof2}
%\begin{split}
\bigg\| x^\alpha\cdot \partial_t \partial^\beta f_{N_i}(x,t) - x^\alpha\cdot \partial^\beta \Big[\nu\Delta f(x,t) + (g(x,t)\cdot\nabla) f(x,t)\\
+ h(x,t)\cdot f(x,t) + k(x,t)\Big]\bigg\|_\infty\ \xrightarrow{i\to\infty}\ 0
%\end{split}
\end{multline}
for all $\alpha,\beta\in\nset_0^n$ uniformly in $t\in [0,T]$. Hence, with $\alpha=\beta=0$ we have that $f$ solves (\ref{eq:generalPDE}) and for every fixed $x\in\rset^n$ the function $G(t) = f(x,t)$ is continuous in $t$. Now let $x\in\rset^n$ and $\alpha=\beta=0$, then (\ref{eq:convergenz}) implies
\begin{subequations}\label{eq:timeRegularity}
\begin{align}
& f(x,t)\notag\\ &= \lim_{i\to\infty} f_{N_i}(x,t)
\intertext{since the $f_{N_i}$ are piece-wise differentiable $\partial_t f_{N_i}(x,t)$ is Riemann integrable in $t$}
&= \lim_{i\to\infty} \int_0^t \partial_t f_{N_i}(x,s)~\diff s
\intertext{where (\ref{eq:convergenz}) implies we can interchange integration and the limit $i\to\infty$}
&= \int_0^t \lim_{i\to\infty} \partial_t f_{N_i}(x,s)~\diff s
\intertext{and then (\ref{eq:proof2}) implies}
&=
\int_0^t \underbrace{\nu\Delta f(x,s) + (g(x,s)\cdot\nabla) f(x,s)+ h(x,s)\cdot f(x,s) + k(x,s)}_{\text{integrand (I)}}~\diff s
\end{align}\end{subequations}
for all $t\in [0,T]$. Since $g$, $h$, and $k$ are $C^d$ in $t$ and $f$ is $C^0$, we have that the integrand (I) is $C^0$ in $t$. Hence, $f$ is an integral over a $C^0$ function and therefore $C^1$ in $t$. Continuing this argument we see that $f$ is $C^{d+1}$ in $t$ since the integrand (I) is at least $C^d$ in $t$.

Since $T>0$ was arbitrary, $\{f_N\}_{N\in\nset}$ has an accumulation point $f$ for all $T>0$ and all accumulation points fulfill (\ref{eq:generalPDE}) with $f$ is $C^{d+1}$ in $t$.
\end{proof}

\begin{rem}
In the standard approach via weak solutions and Sobolev theory one usually works with spaces $L^p([0,T],L^q(\rset^n))$, $1\leq p,q\leq \infty$, or more generally $L^p([0,T],X)$, $X$ a Banach space with norm $\|\cdot\|_X$. Then $L^p([0,T],X)$ is equipped with the $L^p$-norm $\|F\|_{L^p}$ of $F(t):= \|f(\,\cdot\,,t)\|_X$ and therefore it is also a Banach space. For regularity in $t$ one then has additional work to do.

But in our approach we do not work in a Banach space $X$ but in a Montel space, i.e., we do not have a single norm $\|\cdot\|_X$ but a family of semi-norms, here $\|x^\alpha\cdot \partial^\beta f(x,t)\|_\infty$. The convergence of our approximation is (\ref{eq:convergenz}), i.e., uniform on $\rset^n\times [0,T]$ for all derivatives. (\ref{eq:convergenz}) takes care of all spatial derivatives and therefore by (\ref{eq:timeRegularity}) also of the time derivatives. We demonstrated this explicitly in (\ref{eq:timeRegularity}) for clarity but this argument also follows from \cite[Thm.\ 7.17]{rudin76}.\exmsymbol
\end{rem}

%\begin{rem}
%\Cref{thm:main} also applies to the periodic case on $\tset^n$ since \Cref{lem:arzelaascoli} also holds for $C^\infty(\tset^n,\rset)$ functions. In the bounds and in the proof we let $\alpha=0$. We have to check that the functions $f_N$ in \Cref{dfn:fN} are periodic. But the only difficulty might appear in the approximation of $g\cdot\nabla f$. But $g\in C^\infty(\tset^n)$ is also periodic and hence
%
%\begin{multline*}
%f_{N,t_i}(x+1,t) = f_{N,t_{i-1}}\left(x + 1 + \int_{t_{i-1}}^t g(x+1,s)~\diff s \right)\\
%= f_{N,t_{i-1}}\left( x + \int_{t_{i-1}}^t g(x,s)~\diff s\right) = f_{N,t_i}(x,t).
%\end{multline*}
%
%This proves the periodicity.\exmsymbol
%\end{rem}

%\begin{rem}
%The approach in \Cref{thm:main} is used for $C_b^\infty$-coefficients and a generalization is nontrivial. Let e.g.\ $f_0\in\cS(\rset)$ with $f_0(0) = 1$ and we want to approximate with the Montel splitting method a solution of $\partial_t f(x,t) = x^2\cdot \partial_x f(x,t)$. Then the approximation $\tilde{f}(x,t) = f_0(x + t\cdot x^2)$ breaks down immediately since for $t>0$ we have with $x_0 = -t^{-1}$ that
%
%\[\|x\cdot f_0(x + t\cdot x^2)\|_\infty \geq \left| t^{-1}\cdot f_0(-t^{-1} + t^{-1}) \right| = t^{-1} \quad\xrightarrow{t\searrow 0}\quad \infty\]
%
%Hence, we can not apply the Trotter approach in \Cref{dfn:fN}.\exmsymbol
%\end{rem}

\begin{rem}\label{rem:lowerDependencies}
From (\ref{eq:mainBound}) we see that $\|x^\alpha\cdot\partial^\beta f(\,\cdot\,,t)\|_\infty$ depend only on $\|x^\gamma\cdot\partial^\delta f_0\|_\infty$ for all $\gamma,\delta\in\nset_0^n$ with $|\gamma|\leq |\alpha|$ and $|\delta|\leq |\beta|$. Hence, weaker conditions on the initial value $f_0$ is possible since $\cS(\rset^n)$ is dense in any $W^{p,k}(\rset^n)$ with $k\in\nset_0$ and $p\in [1,\infty)$.\exmsymbol
\end{rem}

\section{Burgers' Equation}%%%
%%%%%%%%%%%%%%%%%%%%%%%%%%%%%%
\label{sec:burgers}

For Burgers' equation we have from \Cref{thm:main} the following.

\begin{thm}\label{thm:burger}
Let $u_0\in\cS(\rset,\rset)$. Then there exist maximal $T_1,T_2>0$ such that Burgers' equation
\begin{equation}\label{eq:burgers}\begin{split}
\partial_t u &= -u\cdot\partial_x u\\
u(\,\cdot\,,0) &= u_0
\end{split}\end{equation}
has a unique classical solution $u\in C^\infty((-T_1,T_2),\cS(\rset,\rset))$. $(-T_1,T_2)$ is the maximal interval such that $u\in C((-T_1,T_2),C_b^\infty(\rset,\rset))$.
\end{thm}
\begin{proof}
The $C_b^\infty$ solution of Burgers' equation is unique, i.e., there exists a maximal time interval $(-T_1,T_2)$ such that $u\in C^\infty((-T_1,T_2),C_b^\infty(\rset,\rset))$. Set $n=m=1$, $\nu=0$, $g=u$, $h=0$, and $k=0$ in (\ref{eq:generalPDE}). Then \Cref{thm:main} shows that there exists a $f\in C^\infty((-T_1,T_2),\cS(\rset,\rset))$ that solves (\ref{eq:generalPDE}). By uniqueness of $u$ from Burgers' equation we have $f=u$.
\end{proof}

Since for Burgers' equation we have $u(\,\cdot\,,t)\in\cS(\rset,\rset)$ we can in theory calculate all moments of $u$ for all times $t\in (-T_1,T_2)$. The simplicity of (\ref{eq:burgers}) allows us to calculate the time-dependent moments explicitly.

\begin{thm}\label{thm:burgerMoments}
Let $u_0\in\cS(\rset,\rset)$. Then for all $p\in\nset$ and $k\in\nset_0$ the time-dependent moments
$s_{k,p}(t) := \int_\rset x^k u(x,t)^p\,\diff x$
of the solution $u$ of Burgers' equation (\ref{eq:burgers}) are
\[s_{k,p}(t) = \sum_{i=0}^{k} \frac{s_{k-i,p+i}(0)}{i!}\cdot t^i\cdot \prod_{j=0}^{i-1} \frac{(p+j)\cdot (k-j)}{1+(p+j)^2} \quad\in\rset[t].\]
\end{thm}
\begin{proof}
We proceed via induction over $k\in\nset_0$.

\underline{$k=0$:} We have
\begin{align*}
\partial_t s_{0,p}(t) &= \partial_t \int_\rset u(x,t)^p~\diff x
= -p \int_\rset u(x,t)^p \cdot \partial_x u(x,t)~\diff x
\intertext{with partial integration since $u(\,\cdot\,,t)$ is a Schwartz function}
&= p \int_\rset \partial_x [u(x,t)^p] \cdot u(x,t)~\diff x
= p^2 \int_\rset u(x,t)^p\cdot \partial_x u(x,t)~\diff x\\
&= -p\cdot \partial_t s_{0,p}(t)
\end{align*}
which gives $\partial_t s_{0,p}(t) = 0$ and therefore $s_{0,p}(t) = s_{0,p}(0)$.

\underline{$k-1\to k$:} We have
\begin{align*}
\partial_t s_{k,p}(t) &= \partial_t \int_\rset x^k\cdot u(x,t)^p~\diff x\\
&= -p \int_\rset x^k\cdot u(x,t)^p\cdot \partial_x u(x,t)~\diff x\\
&= p \int_\rset \partial_x ( x^k\cdot u(x,t)^p)\cdot u(x,t)~\diff x\\
&= p\cdot k \int_\rset x^{k-1}\cdot u(x,t)^{p+1}~\diff x + p^2 \int_\rset x^k\cdot u(x,t)^p\cdot \partial_x u(x,t) ~\diff x\\
&= p\cdot k\cdot s_{k-1,p+1}(t) - p^2\cdot\partial_t s_{k,p}(t)\\
&= \frac{p\cdot k}{1+p^2}\cdot s_{k-1,p+1}(t)
\end{align*}
and solving this induction gives
\begin{align*}
s_{k,p}(t) &= s_{k,p}(0) + \frac{p\cdot k}{1+p^2} \int_0^t s_{k-1,p+1}(\tau_1)~\diff\tau_1\\
&= s_{k,p}(0) + \frac{p\cdot k}{1+p^2} \int_0^t \left[ s_{k-1,p+1}(0) + \frac{(p+1)(k-1)}{1+(p+1)^2} \int_0^{\tau_1} s_{k-2,p+2}(\tau_2)~\diff\tau_2\right]\diff\tau_1\\
&\ \;\vdots\\
&= \sum_{i=0}^{k} \frac{s_{k-i,p+i}(0)}{i!}\cdot t^i\cdot \prod_{j=0}^{i-1} \frac{(p+j)\cdot (k-j)}{1+(p+j)^2}
\end{align*}
which proves the statement.
\end{proof}

In Burgers' equation as a transport equation when $u_0\geq 0$ then the classical solution remains non-negative. But from the moments in \Cref{thm:burgerMoments} we observe the following.

\begin{exm}\label{exm:momFinBreakDown}
For $p=1$ we have the following three explicit time-dependent moments from \Cref{thm:burgerMoments}:
\begin{align*}
\int_\rset u(x,t)~\diff x =\ s_{0,1}(t)\ &= s_{0,1}(0),\\
\int_\rset x\cdot u(x,t)~\diff x =\ s_{1,1}(t)\ &= s_{1,1}(0) + s_{0,2}(0)\cdot t,\\
\int_\rset x^2\cdot u(x,t)~\diff x =\ s_{2,1}(t)\ &= s_{2,1}(0) + s_{1,2}(0)\cdot t + \frac{2 s_{0,3}(0)}{5}\cdot t^2.
\end{align*}
For the function
\[u_0(x) := \begin{cases} 1+x & \text{for}\ x\in [-1,0],\\ 1-x & \text{for}\ x\in [0,1],\\ 0 & \text{else}
\end{cases}\]
we have $s_{0,1}(0) = 1$, $s_{1,1}(0)=0$, $s_{0,2}(0) = \frac{2}{3}$, $s_{2,1}(0) = \frac{1}{6}$, $s_{1,2}(0) = 0$, $s_{0,3} = \frac{1}{2}$ and therefore
\begin{equation}\label{eq:negative}
\int_\rset (x-t)^2\cdot u(x,t)~\diff x = L_{s(t)}((x-t)^2) = \frac{1}{6} - \frac{2}{15}t^2 \quad\xrightarrow{t\to \pm\infty}\quad -\infty.
\end{equation}
Since $u_0\not\in\cS(\rset)$ using a mollifier we get $u_0^\varepsilon := S_\varepsilon * u_0 \in C_0^\infty(\rset)\subset \cS(\rset)$ for any $\varepsilon > 0$. We can chose by continuity of the $s_{p,k}(0)$ on $\varepsilon$ an $\varepsilon>0$ small such that the coefficient of $t^2$ in (\ref{eq:negative}) remains negative. Hence, non-negativity in the assumed classical solution is not preserved, i.e., we have a finite breakdown.\exmsymbol
\end{exm}

Let $k\in\nset$ and $k\geq 2$. For
\begin{equation}\label{eq:burgerK}
\partial_t u = u^k\cdot \partial_x u
\end{equation}
multiply (\ref{eq:burgerK}) with $k\cdot u^{k-1}$ to get $\partial_t (u^k) = u^k\cdot \partial_x (u^k)$. This is Burgers' equation with $v=u^k$. If $u_0\geq 0$ we can allow $k\in [1,\infty)$ in (\ref{eq:burgerK}).

\section{Euler and Navier--Stokes Equations}%%%
%%%%%%%%%%%%%%%%%%%%%%%%%%%%%%%%%%%%%%%%%%%%%%%
\label{sec:ens}

By Beale--Kato--Majda \cite{beale84} the classical solutions of the Euler and the Navier--Stokes equations $u$ and $\omega$ exist as long as
$\|\omega(\,\cdot\,,t)\|_\infty <\infty$.
A finite breakdown in time can therefore be observed through a breakdown of $\|\omega\|_\infty$. For the Euler and the Navier--Stokes equations we have from \Cref{thm:main} the following.

\begin{thm}\label{thm:ensSchwartz}
Let $\nu\in [0,\infty)$, $u_0\in C_b^\infty(\rset^3,\rset^3)$ with $\divv u_0=0$ and $\omega_0:=\rot u_0\in\cS(\rset^3,\rset^3)$, and $T^*>0$ be maximal such that $u$ is the solution of the Euler ($\nu=0$) resp.\ Navier--Stokes ($\nu>0$) equation (\ref{eq:ensRot}) with
$\|\omega(\,\cdot\,,t)\|_\infty < \infty$
for all $t\in [0,T^*)$, i.e., the unique smooth solution $u$ exists for all $t\in [0,T^*)$. Then
$\omega \in C^\infty([0,T^*),\cS(\rset^3,\rset^3))$.
\end{thm}
\begin{proof}
The (vorticity formulation of the) Euler and the Navier--Stokes equations have a unique smooth solution $u,\omega\in C^\infty([0,T^*),C_b^\infty(\rset^3,\rset^3))$. Set $n=m=3$, $g=u$, $h=\nabla u$, and $k=0$ in (\ref{eq:generalPDE}). Then by \Cref{thm:main} we have a solution $f\in C^\infty([0,T^*),\cS(\rset^3,\rset^3))$. But by the uniqueness of $\omega$ we have $f=\omega$.
\end{proof}

A breakdown in $\|\omega\|_\infty$ provides $T^*<\infty$ \cite{beale84}. Now by \Cref{thm:ensSchwartz} a breakdown in any
$\|x^\alpha\cdot \partial^\beta \omega^\gamma\|_\infty$ or $\|x^\alpha\cdot \partial^\beta \omega^\gamma\|_{L^p(\rset^3)}$
with $p\in [1,\infty)$, $\alpha,\beta,\gamma\in\nset_0^3$, and $\omega^\gamma := \omega_1^{\gamma_1}\cdot \omega_2^{\gamma_2}\cdot \omega_3^{\gamma_3}$, $\gamma=(\gamma_1,\gamma_2,\gamma_3)\neq 0$, also provides $T^*<\infty$. By \Cref{rem:lowerDependencies} weaker conditions on $\omega_0$ are possible. Unfortunately, similar calculations as in \Cref{thm:burgerMoments} or \Cref{exm:momFinBreakDown} for Burgers' equation are not yet accessible to us for the Euler or Navier--Stokes equations. With $k = \rot F\in C^\infty([0,\infty),\cS(\rset^3,\rset^3))$ in \Cref{thm:main} we have that \Cref{thm:ensSchwartz} also holds with external forces.

\section*{Acknowledgments}%%%
%%%%%%%%%%%%%%%%%%%%%%%%%%%%%

We thank Tarek Elgindi for valuable remarks and fruitful discussions on the paper. We thank Lorenzo Brandolese for valuable remarks, fruitful discussions, and for providing additional references.

The author and this project are financed by the Deutsche Forschungs\-gemein\-schaft DFG with the grant DI-2780/2-1 and his research fellowship at the Zukunfs\-kolleg of the University of Konstanz, funded as part of the Excellence Strategy of the German Federal and State Government.

{\addcontentsline{toc}{section}{References}
%%\section*{References}
%%\bibliographystyle{amsplain}
%%\bibliographystyle{amsalpha}
%\bibliography{../../bibdata}

\end{document}